\documentclass[11pt]{amsart}
\usepackage{amssymb,amsmath}
\usepackage[colorlinks]{hyperref}
\usepackage{graphicx}
\usepackage[british]{babel}
\usepackage{microtype} 
\usepackage{tikz}
\usepackage{pgfplots}
\pgfplotsset{compat=1.18}

\makeatletter
\theoremstyle{plain}
 \newtheorem{thm}{Theorem}[section]
\newtheorem{thm*}{Theorem}
 \newtheorem{lem}[thm]{Lemma}
 \newtheorem{prop}[thm]{Proposition}
 
 \numberwithin{equation}{section}
\numberwithin{figure}{section} 
 \theoremstyle{plain}
 \theoremstyle{definition}
 \newtheorem{defn}[thm]{Definition}
 \newtheorem{rem}[thm]{Remark}

\newcommand{\calM}{{{\mathcal M}}}

\newcommand{\calN}{{{\mathcal N}}}

\newcommand{\C}{{{\mathbb C}}}
\newcommand{\R}{{{\mathbb R}}}

\newcommand{\bL}{{{\bf L}}}
\newcommand{\bX}{{{\bf X}}}
\newcommand{\bY}{{{\bf Y}}}
\newcommand{\bZ}{{{\bf Z}}}
\newcommand{\bU}{{{\bf U}}}
\newcommand{\bV}{{{\bf V}}}
\newcommand{\bW}{{{\bf W}}}

\newcommand{\bz}{{{\bf z}}}
\newcommand{\bw}{{{\bf w}}}

\newcommand{\fL}{{{\mathfrak L}}}

\makeatother

\date{\today\\
2020 \emph{Mathematics Subject Classifications.} 32Q55, 53B99.\\
\emph{Key words.} Real-linear maps, CR geometry, rank-one matrices.\\
}
\begin{document}

\title[${\rm SPCR}$ automorphisms] {Strictly paired ${\rm CR}$ linear automorphisms of $\R^4$}

\author[I.D. Platis]{Ioannis D. Platis}

\begin{abstract}
This paper studies real-linear automorphisms of $\mathbb C^2$ through the
ranks of their complex-linear and conjugate-linear blocks. Particular emphasis
is placed on strictly paired CR (SPCR) automorphisms of $\mathbb R^4$, for
which both blocks have rank $1$. Such a map is complex-linear on one complex
line and conjugate-linear on a complementary complex line. Equivalently, the
SPCR automorphisms form the double coset
$\operatorname{GL}(2,\mathbb C)\,\sigma_0\,\operatorname{GL}(2,\mathbb C)$,
where $\sigma_0(z_1,z_2)=(z_1,\bar z_2)$ is a mixed real-linear involution
whose fixed locus is a real hypersurface. We show that the SPCR locus is a
$12$-dimensional submanifold of $M_2(\mathbb C)^2$ and carries a CR structure
of CR dimension $4$ and CR codimension $4$ on an explicit open dense subset.
We also construct canonical diagonal and genuinely non-diagonal examples.
\end{abstract}

\address{Department of Mathematics, University of Patras, Patras, Greece}
\email{idplatis@upatras.gr}

\maketitle

\section{Introduction}

The study of smooth mappings between domains in $\C^2$ naturally leads to the
classification of their differentials according to the ranks of their holomorphic
and antiholomorphic components (see, for instance, \cite{Webb} for related
mapping problems). In this paper, we describe the possible pointwise rank
patterns and investigate the properties of strictly paired CR (SPCR) linear
automorphisms.

Let $D_1,D_2$ be open subsets of $\C^2$ and let $f:D_1\to D_2$ be a
diffeomorphism given by
$$
f(z_1,z_2)=(w_1(z_1,z_2),w_2(z_1,z_2)),\quad (z_1,z_2)\in D_1.
$$
Fix $p\in D_1$ and set $q=f(p)$. The derivative
$f_{*,p}:T_{p}(D_1)\to T_{q}(D_2)$ has matrix
\begin{equation*}
Df(p)=\left(\begin{matrix}
A&B\\
\overline{B}&\overline{A}
\end{matrix}\right)_{p},
\end{equation*}
where
$$
A=\left(\begin{matrix}
\frac{\partial w_1}{\partial z_1}&\frac{\partial w_1}{\partial z_2}\\
\\
\frac{\partial w_2}{\partial z_1}&\frac{\partial w_2}{\partial z_2}
\end{matrix}\right),\quad B=\left(\begin{matrix}
\frac{\partial w_1}{\partial \overline{z_1}}&\frac{\partial w_1}{\partial \overline{z_2}}\\
\\
\frac{\partial w_2}{\partial \overline{z_1}}&\frac{\partial w_2}{\partial \overline{z_2}}
\end{matrix}\right).
$$
Since $f$ is a diffeomorphism, ${\rm rank}(Df(p))=4$ at every $p\in D_1$.
The following list gives the possible rank pairs at a fixed point; we suppress
the dependence on $p$. If one of these rank patterns holds at every point,
then the corresponding description applies throughout $D_1$.

\begin{enumerate}
\item ${\rm rank}(A)=2$ and ${\rm rank}(B)=0$. Then $f$ is holomorphic by
the Cauchy--Riemann equations and by the Inverse Function Theorem $f$ is also a
biholomorphism. The isomorphism of the holomorphic tangent spaces
$T^{(1,0)}_p(D_1)$ and $T^{(1,0)}_q(D_2)$ is given by $f_{*,p}$; if
$Z=\sum_{i=1}^2a_i(\partial/\partial z_i)_p$ is a $(1,0)$-vector then
$$
f_{*,p}(Z)=\sum_{i=1}^2\sum_{j=1}^2a_i\frac{\partial w_j}{\partial z_i}(p)\left(\frac{\partial}{\partial w_j}\right)_q.
$$
Note that by conjugation we also have that $f_{*,p}$ identifies
$T^{(0,1)}_p(D_1)$ and $T^{(0,1)}_q(D_2)$ as well. As for the Jacobian
determinant of $Df(p)$ we have
$$
\det(Df(p))=\left|\begin{matrix}
A&0\\
0&\overline{A}
\end{matrix}\right|_p=|\det(A)_p|^2>0.
$$
Geometrically speaking, for all $p\in D_1$ the derivative $f_{*,p}$ maps
complex lines in $T_p(D_1)$ onto a complex line in $T_q(D_2)$; the orientation
of complex lines is preserved and $f$ is orientation-preserving.

\item ${\rm rank}(A)=0$ and ${\rm rank}(B)=2$. This is the counterpart of the
previous case: Here, $f$ is an anti-biholomorphism; $f_{*,p}$ identifies
$T^{(1,0)}_p(D_1)$ to $T^{(0,1)}_q(D_2)$. Explicitly, given a $(1,0)$-vector
$Z=\sum_{i=1}^2a_i(\partial/\partial z_i)_p$, then
$$
f_{*,p}(Z)=\sum_{i=1}^2\sum_{j=1}^2a_i\frac{\partial \overline{w_j}}{\partial z_i}(p)\left(\frac{\partial}{\partial \overline{w_j}}\right)_q.
$$
Again by conjugation, $f_{*,p}$ also identifies $T^{(0,1)}_p(D_1)$ and
$T^{(1,0)}_q(D_2)$. Note though that in contrast with what happens in the
one-dimensional case, here the antiholomorphic $f$ is {\it
orientation-preserving}:
$$
\det(Df(p))=\left|\begin{matrix}
0&B\\
\overline{B}&0
\end{matrix}\right|_p=\left|\begin{matrix}
B&0\\
0&\overline{B}
\end{matrix}\right|_p=|\det(B)_p|^2>0.
$$
Here, for all $p\in D_1$ the derivative $f_{*,p}$ maps complex lines in
$T_p(D_1)$ onto a complex line in $T_q(D_2)$; the orientation of complex lines
is reversed whereas $f$ is again orientation-preserving.

\item ${\rm rank}(A)={\rm rank}(B)=2$. This is the fully mixed case: both
the complex-linear and conjugate-linear components are nonzero on every
nonzero complex direction. Thus no complex line in the source is a direction
on which the restriction of $f_{*,p}$ is purely complex-linear or purely
conjugate-linear. This does not, however, preclude the real image of a complex
line from being a complex line. The real orientation may be preserved or
reversed, depending on the map.

\item ${\rm rank}(A)=2$ and ${\rm rank}(B)=1$. Then there is exactly one
complex-linear source direction, namely $\overline{\ker B}$. This is the
pointwise rank pattern associated with a CR-type differential when that line
is the complex tangent direction. The rank condition alone does not specify a
hypersurface or assert that $f$ maps one given hypersurface to another.

\item ${\rm rank}(A)=1$ and ${\rm rank}(B)=2$. Then there is exactly one
conjugate-linear source direction, namely $\ker A$. This is the corresponding
anti-CR-type rank pattern; again, a hypersurface and its complex tangent have
to be specified before one can assert a CR mapping property.

\item ${\rm rank}(A)={\rm rank}(B)=1$. We call this the strictly paired
CR (SPCR) rank pattern: it gives one complex-linear and one
conjugate-linear source direction, and these directions are complementary
when the real differential is invertible.
\end{enumerate}
The six cases are not merely a taxonomic list: the ranks of $A$ and $B$ count
distinguished complex directions associated with the differential. The pairs
$(0,0)$, $(0,1)$ and $(1,0)$ are excluded by invertibility of $Df(p)$. At the fixed point $p$, write $L:=f_{*,p}$. For $v\in\C^2$ and
$\zeta\in\C$ one has the identity
\begin{equation}\label{eq:directional}
L(\zeta v)=\zeta Av+\bar\zeta B\bar v.
\end{equation}
Thus $L$ is \emph{complex-linear} on the complex line $\C\cdot v$ precisely
when $B\bar v=0$, i.e.\ $v\in\overline{\ker B}$, and
\emph{anti-complex-linear} on $\C\cdot v$ precisely when $Av=0$, i.e.\
$v\in\ker A$. Consequently:
\begin{itemize}
\item if $B=0$, every complex line is a complex-linear direction, and $L$ is
holomorphic;
\item if $\operatorname{rank}B=1$, the unique complex-linear direction is the
line $\overline{\ker B}$;
\item if $B$ is invertible, $L$ has no complex-linear direction;
\end{itemize}
and symmetrically for $A$ and conjugate-linear directions. The first two
patterns are the pure holomorphic and pure antiholomorphic cases. When both
blocks are invertible there are no purely complex-linear or purely
conjugate-linear source directions, although a mixed real-linear image of a
complex line can still happen to be a complex line. The next two patterns
have exactly one distinguished direction of the corresponding type. They are
pointwise CR-type and anti-CR-type patterns, not by themselves assertions
about maps between specified hypersurfaces.

The sixth case, which is the subject of this paper, is the unique case in which both
a complex-linear and an anti-complex-linear direction coexist. For an SPCR
automorphism these directions are automatically transverse: if
$\overline{\ker B}\cap\ker A\neq 0$, then by \eqref{eq:directional} $L$ would
have a non-trivial kernel, contradicting invertibility; and the images
$\operatorname{im}A$ and $\operatorname{im}B$ are distinct for the same
reason. Hence
\[
\C^2=\overline{\ker B}\oplus\ker A=\operatorname{im}A\oplus\operatorname{im}B,
\]
and $L$ acts complex-linearly on the first factor and anti-complex-linearly on
the second. SPCR automorphisms are therefore the \emph{minimal mixed case}:
the smallest-rank pattern in which both the holomorphic and the
anti-holomorphic parts survive, and the only one in which they act on
complementary complex directions.

For the linear discussion, let $\calM\subset\operatorname{GL}(4,\mathbb R)$
denote the set of all real-linear automorphisms for which
$\operatorname{rank}A=\operatorname{rank}B=1$. The transversality just
described allows us to put every element of $\calM$ into a canonical
diagonal form. We state this as the first main result of the paper.

\begin{thm}[Canonical Form]\label{thm:canonical}
Every SPCR automorphism $L\in\calM$ is
$\operatorname{GL}(2,\C)$-equivalent, under left and right composition by
biholomorphisms of $\C^2$, to the diagonal form
\begin{equation}\label{eq:canonical}
(z_1,z_2)\longmapsto(az_1,\,b\bar z_2),\qquad a,b\in\C^*.
\end{equation}
Equivalently,
\[
\calM=\operatorname{GL}(2,\C)\,\sigma_0\,\operatorname{GL}(2,\C),
\qquad
\sigma_0(z_1,z_2)=(z_1,\bar z_2).
\]
\end{thm}

\begin{proof}
By definition, an SPCR automorphism has
$\operatorname{rank}A=\operatorname{rank}B=1$. Invertibility of $L$ forces
$\overline{\ker B}\cap\ker A=0$: for $v\in\overline{\ker B}\cap\ker A$ one
has $B\bar v=0$ and $Av=0$, hence $L(v)=Av+B\bar v=0$, so $v=0$. Since both
$\overline{\ker B}$ and $\ker A$ are complex lines in $\C^2$, they are
complementary. Likewise, $\operatorname{im}A\cap\operatorname{im}B=0$: if $\operatorname{im}A=\operatorname{im}B=\C w$, then the image of the real-linear map $L$ would be contained in the real $2$-plane underlying $\C w$, contradicting invertibility.

Choose $M,N\in\operatorname{GL}(2,\C)$ such that
$M(\operatorname{im}A)=\C e_1$, $M(\operatorname{im}B)=\C e_2$,
$N^{-1}(\overline{\ker B})=\C e_1$, and $N^{-1}(\ker A)=\C e_2$. Set
$L':=M\circ L\circ N$. A direct computation gives $A'=MAN$, $B'=MB\bar N$,
and one verifies:
\[
\operatorname{im}A'=M(\operatorname{im}A)=\C e_1,\quad
\operatorname{im}B'=M(\operatorname{im}B)=\C e_2,
\]
\[
\ker A'=N^{-1}(\ker A)=\C e_2,\quad
\overline{\ker B'}=\C e_1.
\]
Hence
$A'=\begin{pmatrix}a&0\\0&0\end{pmatrix}$ and
$B'=\begin{pmatrix}0&0\\0&b\end{pmatrix}$ for some $a,b\in\C^*$.
Both scalars are nonzero because $\operatorname{rank}A'=\operatorname{rank}B'=1$.
Therefore
\[
L'(z_1,z_2)=A'z+B'\bar z=(az_1,b\bar z_2).
\]
Conversely, the map in \eqref{eq:canonical} has both blocks of rank $1$
and has determinant $-|a|^2|b|^2<0$. Hence it is SPCR. Left and right
composition by complex-linear automorphisms preserves the two block ranks
and invertibility, so the canonical form is indeed equivalent to every
element of $\calM$.
\end{proof}

The canonical form identifies $\calM$ with the double coset
\[
\operatorname{GL}(2,\C)\,\sigma_0\,\operatorname{GL}(2,\C),
\]
where $\sigma_0(z_1,z_2)=(z_1,\bar z_2)$. Notice that $\sigma_0$ is a
real-linear involution, but it is neither complex-linear nor
complex-antilinear on all of $\C^2$: it preserves the complex structure on
the first coordinate and reverses it on the second. Its fixed-point set is
\[
\operatorname{Fix}(\sigma_0)=\{(z_1,x_2):z_1\in\C,\ x_2\in\R\}=\C\times\R,
\]
a real hypersurface. Thus the term ``paired'' refers to the coexistence of
complex-linear and conjugate-linear behaviour in complementary directions,
not to an anti-linear involution on all of $\C^2$.

Finally, we make here two further remarks in order to explain why the SPCR class is worth studying. Both are motivational only: no result of this paper depends on them, and the
corresponding questions are collected in \S\ref{sec:outlook}.

\smallskip\noindent
\emph{Mixed-type maps.} The canonical representative is holomorphic in one
coordinate and antiholomorphic in the other. It is a linear mixed-type model;
we do not assert that it maps any specified real hypersurface to another as a
CR map. Questions about nonlinear mixed-type maps may be compared with
classical reflection and mapping problems for CR manifolds
\cite{Boggess, BER, Webb, Pinchuk}, but no reflection-principle result is
claimed here.

\smallskip\noindent
\emph{A hypersurface-fixing involution.} The map
$\sigma_0(z_1,z_2)=(z_1,\bar z_2)$ is a mixed real-linear involution with
fixed locus $\C\times\R$. The double-coset description in
Theorem~\ref{thm:canonical} is a statement about this involution and the
left and right actions of $\operatorname{GL}(2,\C)$. This construction
suggests possible, but presently undeveloped, connections with para-CR
geometry.

The main results and the organisation of this paper are as follows:
Section~\ref{sec:linear} sets up the algebraic identification of
$\operatorname{GL}(4,\R)$ with a real Lie subgroup of $\operatorname{GL}(4,\C)$ and
proves the determinant formula (Lemma~\ref{lem:detLformula}), the rank-one
characterisation (Lemma~\ref{lem:detL}), and the identification of the SPCR
locus with the inverse image $F^{-1}(0,0)$ of the smooth map
$F(M,N)=(\det(M+\bar N),\det(M-\bar N))$ (Proposition~\ref{prop:SPCReqs}).
Proposition~\ref{prop:dim12} shows that $\calM$ is a $12$-dimensional
submanifold of $M_2(\C)^2$, and Proposition~\ref{prop:CR2} identifies a
CR structure of CR dimension $4$ and CR codimension $4$ on an explicit open
dense subset $U\subset\calM$ (see \S\ref{ss:CR} for the precise statement).
Section~\ref{sec:examples} works out the canonical diagonal family with every
condition of Lemma~\ref{lem:detL} verified by hand, exhibits a genuinely
non-diagonal SPCR automorphism together with its point on the $12$-dimensional
manifold, and uses a counterexample to show that the hypothesis
$\det\mathbf L\neq 0$ in Proposition~\ref{prop:SPCReqs} is indispensable.
Section~\ref{sec:outlook} collects several open problems.

\section{${\rm SPCR}$ Linear Automorphisms of $\C^2$}\label{sec:linear}

In this section, we establish the algebraic framework for linear automorphisms
of $\R^4$ by identifying them with a specific subgroup of block matrices in
${\rm GL}(4,\C)$. Building upon this identification, we formally introduce
Strictly Paired CR (SPCR) linear automorphisms, characterizing them through
exact determinant constraints on their complex matrix components. Finally, we show that the SPCR locus is a $12$-dimensional real
submanifold of $M_2(\C)^2$, and identify its CR structure on an open dense
subset.

\subsection{Linear Automorphisms of $\R^4$}
The set of linear automorphisms of
$$\R^4=\{(x_1,x_2,y_1,y_2)\;|\;x_i,y_i\in\R,\;i=1,2\},
$$
is the Lie group ${\rm GL}(4,\R)$ comprising real $4\times 4$ matrices with
non-zero determinant. If $L\in{\rm GL}(4,\R)$ we may write
$$
L=\left[\begin{matrix}
                                     A_1&A_2\\
                                     A_3&A_4
                                    \end{matrix}\right],
$$
where $A_i$, $i=1,\dots, 4$ are $2\times 2$ real matrices. There is an
identification of ${\rm GL}(4,\R)$ with a real Lie subgroup of the complex Lie
group ${\rm GL}(4,\C)$ comprising complex $4\times 4$ matrices with non-zero
determinant; we describe this identification below. Let
$$\C^2=\{(z_1,z_2)\;|\;z_i=x_i+iy_i\in\C,\;i=1,2\}
$$
and let also
$$
\bX=\left[\begin{matrix} x_1\\x_2
                  \end{matrix}\right],\quad
\bY=\left[\begin{matrix} y_1\\y_2
                  \end{matrix}\right],\quad
                  \bU=\left[\begin{matrix} u_1\\u_2
                  \end{matrix}\right],\quad
                  \bV=\left[\begin{matrix} v_1\\v_2
                  \end{matrix}\right],\quad
                  \bZ=\left[\begin{matrix} z_1\\z_2
                  \end{matrix}\right],\quad
                  \bW=\left[\begin{matrix} w_1\\w_2
                  \end{matrix}\right],
$$
where $w_i=u_i+iv_i$, $i=1,2$. For an arbitrary $L\in {\rm GL}(4,\R)$ let
the corresponding linear mapping $L:\R^4\to\R^4$. This is written in matrix
form as
\begin{eqnarray*}
 L\left(\left[\begin{matrix}
               \bX\\
               \bY
              \end{matrix}\right]\right)=\left[\begin{matrix}
               \bU\\
               \bV
              \end{matrix}\right]&=&\left[\begin{matrix}
                                     A_1&A_2\\
                                     A_3&A_4
                                    \end{matrix}\right]\left[\begin{matrix}
               \bX\\
               \bY
              \end{matrix}\right].
\end{eqnarray*}
 Writing the previous equation as
\begin{eqnarray*}
 &&
 \bU=A_1\bX+A_2\bY,\\
 &&
 \bV=A_3\bX+A_4\bY,
\end{eqnarray*}
and using the relations
$$
\bX=\frac{1}{2}(\bZ+\overline{\bZ}),\quad \bY=\frac{1}{2i}(\bZ-\overline{\bZ}),\quad
\bU=\frac{1}{2}(\bW+\overline{\bW}),\quad \bV=\frac{1}{2i}(\bW-\overline{\bW}),
$$
we have
\begin{eqnarray*}
 &&
 \bW+\overline{\bW}=(A_1-iA_2)\bZ+(A_1+iA_2)\overline{\bZ},\\
 &&
 \bW-\overline{\bW}=(A_4+iA_3)\bZ-(A_4-iA_3)\overline{\bZ}.
\end{eqnarray*}
By adding the above relations we obtain the equation of the linear mapping $L$
in complex form:
\begin{equation*}
\bW=L(\bZ)=A\bZ+B{\overline \bZ},
\end{equation*}
where
\begin{align}
A&=\frac{1}{2}\left((A_1+A_4)+i(A_3-A_2)\right)
:=\begin{pmatrix} a_1&a_2\\ a_3&a_4\end{pmatrix},\label{eq:AB}\\
B&=\frac{1}{2}\left((A_1-A_4)+i(A_3+A_2)\right)
:=\begin{pmatrix} b_1&b_2\\ b_3&b_4\end{pmatrix}.\label{eq:AB2}
\end{align}
Here $A,B\in M_2(\C)$.
The map $L$ is an isomorphism (and hence a diffeomorphism) if and only if the
$4\times 4$ linear system
\begin{eqnarray}
&&\label{eq:cs1}
 \bW=A\bZ+B\overline{\bZ},\\
 &&\label{eq:cs2}
 \overline{\bW}=\overline{B}\bZ+\overline{A}\overline{\bZ},
\end{eqnarray}
admits a unique solution. That is equivalent to saying that the complex matrix
\begin{equation}\label{eq:matL}
{\bf L}=\left(\begin{matrix}
               A&B\\
               \overline{B}&\overline{A}
              \end{matrix}\right)\in {\rm GL}(4,\C).
\end{equation}
Conversely, if (\ref{eq:matL}) holds, then the system defined by
(\ref{eq:cs1}) and (\ref{eq:cs2}) admits a unique solution. By adding and
subtracting these two equations we obtain
\begin{eqnarray*}
&&
2\bU=\bW+\overline{\bW}=(A+\overline{B})\bZ+(B+\overline{A})\overline{\bZ},\\
&&
2i\bV=\bW-\overline{\bW}=(A-\overline{B})\bZ+(B-\overline{A})\overline{\bZ}.
\end{eqnarray*}
Equivalently, there exists a unique solution to the system
\begin{eqnarray*}
&&
\bU=\Re\left((A+\overline{B})\bZ\right)=\Re\left((A+\overline{B})\right)\bX-
\Im\left((A+\overline{B})\right)\bY,\\
&&
\bV=\Im\left((A-\overline{B})\bZ\right)=\Im\left((A-\overline{B})\right)\bX+
\Re\left((A-\overline{B})\right)\bY,
\end{eqnarray*}
that is
$$
L=\left(\begin{matrix}
\Re\left((A+\overline{B})\right)&-\Im\left((A+\overline{B})\right)\\
\Im\left((A-\overline{B})\right)&\Re\left((A-\overline{B})\right)
\end{matrix}\right)\in{\rm GL}(4,\R).
$$
The change from real coordinates $(\bX,\bY)$ to $(\bZ,\overline{\bZ})$
is an invertible complex change of basis after complexification. Consequently
$\mathbf L$ is similar to the complexification of the real matrix of $L$;
in particular,
$\det_{\mathbb C}(\mathbf L)=\det_{\mathbb R}(L)$.
Note that we always have
$$
\det{\bL}=\left|\begin{matrix}
               A&B\\
               \overline{B}&\overline{A}
              \end{matrix}\right|=
              \left|\begin{matrix}
               B&A\\
               \overline{A}&\overline{B}
              \end{matrix}\right|=
              \left|\begin{matrix}
               \overline{A}&\overline{B}\\
               B&A
              \end{matrix}\right|=
\overline{\det{\bL}}.
$$
If $L\in{\rm GL}(4,\R)$, its rank pair must belong to
\[
(\operatorname{rank}A,\operatorname{rank}B)
\in\{(2,0),(0,2),(2,1),(1,2),(2,2),(1,1)\}.
\]
Indeed, if one block has rank zero, invertibility forces the other block
to have rank two. If neither block has rank zero and both have rank at most
one, then both have rank one. In all other cases at least one block has rank
two. When $B=0$, the map is a linear biholomorphism and
$\det\mathbf L=|\det A|^2>0$; when $A=0$, it is an
antiholomorphic automorphism and $\det\mathbf L=|\det B|^2>0$.
The rank-$(1,1)$ case is the SPCR case studied below.

\medskip

We record the following proposition.
\begin{prop}\label{prop:L-GL}
The image of ${\rm GL}(4,\R)$ under the representation
$L(z)=Az+B\bar z$ is a $16$-dimensional real Lie subgroup of
${\rm GL}(4,\C)$. The biholomorphic automorphisms form a subgroup
isomorphic, as a real Lie group, to ${\rm GL}(2,\C)$; the
antiholomorphic automorphisms form a coset of this subgroup. Both subsets
have real dimension $8$.
\end{prop}
\begin{proof}
The assignment from a real $4\times4$ matrix to its pair $(A,B)$ is an
invertible real-linear change of coordinates, as the formulas
\eqref{eq:AB}--\eqref{eq:AB2} show. It identifies ${\rm GL}(4,\R)$ with
the open subset $\det\mathbf L\ne0$ of the $16$-dimensional real vector
space of matrices of the form
$\left(\begin{smallmatrix}A&B\\\overline B&\overline A\end{smallmatrix}\right)$.
The representation respects composition and inverses, so its image is a
real Lie subgroup. The holomorphic maps are exactly those with $B=0$ and
$A\in{\rm GL}(2,\C)$; the antiholomorphic maps are exactly those with
$A=0$ and $B\in{\rm GL}(2,\C)$. The latter set is a coset of the former
by the map $z\mapsto\bar z$, and both are real $8$-dimensional.
\end{proof}

\begin{rem}\label{rem:directional}
For $L\in\fL_\R$ with matrix $\mathbf L$ as in \eqref{eq:matL}, and for
$v\in\C^2$, the identity $L(\zeta v)=\zeta Av+\bar\zeta B\bar v$ shows that
$L$ is complex-linear on $\C\cdot v$ if and only if $v\in\overline{\ker B}$,
and anti-complex-linear on $\C\cdot v$ if and only if $v\in\ker A$. In
particular, an SPCR automorphism has exactly one complex-linear direction
$\overline{\ker B}$ and exactly one anti-complex-linear direction $\ker A$,
and these are transverse; in suitable complex-linear coordinates $L$ takes
the normal form $L(z_1,z_2)=(az_1,b\bar z_2)$ with $a,b\in\C^*$.
\end{rem}

\subsection{${\rm SPCR}$ Linear Automorphisms of $\R^4$}
Recall that if
$$
A=\left(\begin{matrix}
a_1&a_2\\
a_3&a_4\end{matrix}\right)
$$
is a matrix with complex entries, then its $\text{adjugate}$ is
$$
{\rm adj}(A)=\left(\begin{matrix}
a_4&-a_2\\
-a_3&a_1\end{matrix}\right).
$$
\begin{lem}
If $A, B$ are $2\times 2$ complex matrices, then
\begin{equation}\label{eq:sumdets}
\det(A+B)=\det(A)+{\rm tr}(A{\rm adj}(B))+\det(B).
\end{equation}
\end{lem}
Since ${\rm tr}(A{\rm adj}(B))={\rm tr}(B{\rm adj}(A))$, we can rewrite
\eqref{eq:sumdets} as
\begin{equation}\label{eq:sumdets2}
\det(A+B)=\det(A)+{\rm tr}(B{\rm adj}(A))+\det(B).
\end{equation}

Throughout, $\overline X$ denotes entrywise complex conjugation. The trace
expression in the next lemma is real, as follows from the determinant identity;
no non-negativity assertion is needed.

\begin{lem}\label{lem:detLformula}
Let $L$ be a real-linear transformation of $\R^4$ with complex matrix
${\bf L}$ as in \eqref{eq:matL}. Then
\begin{equation}\label{eq:detLformula-correct}
\det({\bf L})=|\det A|^2+|\det B|^2-
\operatorname{tr}\!\left(\overline A\,\operatorname{adj}(B)
A\,\operatorname{adj}(\overline B)\right).
\end{equation}
In particular, the trace on the right-hand side is real.
\end{lem}
\begin{proof}
First suppose that $A$ is invertible. By the Schur-complement formula,
\[
\det({\bf L})=\det(A)\det\bigl(\overline A-\overline B A^{-1}B\bigr).
\]
For $2\times2$ matrices $C,D$, the identity
\[
\det(C-D)=\det C+\det D-\operatorname{tr}\bigl(C\operatorname{adj}(D)\bigr)
\]
follows by expanding the determinant. Apply it with
$C=\overline A$ and $D=\overline B A^{-1}B$. Since determinants are
multiplicative and $\operatorname{adj}(XY)=\operatorname{adj}(Y)
\operatorname{adj}(X)$ for $2\times2$ matrices, we have
\[
\det(D)=\frac{|\det B|^2}{\det A},\qquad
\operatorname{adj}(D)=\operatorname{adj}(B)\,
\operatorname{adj}(A^{-1})\,\operatorname{adj}(\overline B)
=\frac{1}{\det A}\operatorname{adj}(B)A\operatorname{adj}(\overline B).
\]
Substitution gives \eqref{eq:detLformula-correct}. Both sides of that
identity are polynomial functions of the entries of $A,B$ and their complex
conjugates. Since the identity holds on the dense set $\det A\ne0$, it holds
for all $A,B$. The final assertion follows because $\det({\bf L})$ and the
first two terms on the right are real.
\end{proof}

The following elementary lemma is basic for our subsequent discussion.
\begin{lem}\label{lem:detL}
Let $L$ be a real-linear transformation of $\R^4$ with complex matrix
${\bf L}$ as in \eqref{eq:matL}, and assume that $\det({\bf L})\ne0$.
Then
\[
\operatorname{rank}(A)=\operatorname{rank}(B)=1
\quad\Longleftrightarrow\quad
\det A=\det B=0.
\]
Whenever these equivalent conditions hold, $\det({\bf L})<0$. More
precisely, if $A=u v^T$ and $B=s t^T$, with nonzero column vectors
$u,v,s,t\in\C^2$, then
\begin{equation}\label{eq:detL-rankone}
\det({\bf L})=-\left|\det\begin{pmatrix}u_1&s_1\\u_2&s_2\end{pmatrix}\right|^2
\left|\det\begin{pmatrix}v_1&v_2\\\overline{t_1}&\overline{t_2}\end{pmatrix}\right|^2<0.
\end{equation}
\end{lem}
\begin{proof}
Suppose first that $\operatorname{rank}(A)=\operatorname{rank}(B)=1$.
Choose nonzero column vectors $u,v,s,t\in\C^2$ such that
$A=uv^T$ and $B=st^T$. Then
\[
{\bf L}=
\underbrace{\begin{pmatrix}
u_1&s_1&0&0\\
u_2&s_2&0&0\\
0&0&\overline{s_1}&\overline{u_1}\\
0&0&\overline{s_2}&\overline{u_2}
\end{pmatrix}}_{R}
\underbrace{\begin{pmatrix}
v_1&v_2&0&0\\
0&0&t_1&t_2\\
\overline{t_1}&\overline{t_2}&0&0\\
0&0&\overline{v_1}&\overline{v_2}
\end{pmatrix}}_{S}.
\]
A direct determinant computation gives
\[
\det R=-\left|\det\begin{pmatrix}u_1&s_1\\u_2&s_2\end{pmatrix}\right|^2,
\qquad
\det S=\left|\det\begin{pmatrix}v_1&v_2\\
\overline{t_1}&\overline{t_2}\end{pmatrix}\right|^2.
\]
Thus \eqref{eq:detL-rankone} follows. Since ${\bf L}$ is invertible,
its determinant is nonzero, so the right-hand side of
\eqref{eq:detL-rankone} is strictly negative.

Conversely, assume that $\det A=\det B=0$ and $\det({\bf L})\ne0$.
Each of $A,B$ has rank at most $1$. Neither can be zero: if, for example,
$A=0$, then
${\bf L}=\left(\begin{smallmatrix}0&B\\\overline B&0\end{smallmatrix}\right)$
would be singular because $\det B=0$; the same argument applies if $B=0$.
Consequently both matrices have rank exactly $1$, as required.
\end{proof}

\begin{rem}\label{rem:old-rank-criterion}
The determinant conditions
\[
\det(A+\overline B)+\det(A-\overline B)=0,
\qquad
\det({\bf L})=-|\det(A+\overline B)|^2
\]
that appeared in an earlier version are not equivalent to the rank-one
condition. The first equation is only
$2(\det A+\overline{\det B})=0$. Moreover, the second equation is not
valid for all rank-one pairs. For example,
\[
A=\begin{pmatrix}1&0\\ i&0\end{pmatrix},\qquad
B=\begin{pmatrix}0&1\\0&-i\end{pmatrix}
\]
have rank $1$ and give $\det({\bf L})=-4$ but
$\det(A+\overline B)=0$. Conversely,
$A=I_2$ and $B=\operatorname{diag}(2,-\tfrac12)$ have rank $2$ each,
while the two displayed determinant conditions both hold. Thus a sign change
alone does not repair the previous criterion; the correct equations are
$\det A=\det B=0$.
\end{rem}

\begin{defn}\label{defn:SPCR}
Let $\fL_\R:={\rm GL}(4,\R)$. A ${\rm SPCR}$ linear automorphism $L$ of
$\R^4$ is an element $L\in\fL_\R$ whose associated complex matrix
$\mathbf L$ as in \eqref{eq:matL} satisfies
\[
\operatorname{rank}(A)=\operatorname{rank}(B)=1.
\]
Equivalently, by Lemma~\ref{lem:detL}, it satisfies
$\det A=\det B=0$, with invertibility of $L$ included in the definition.
We use $\calM$ for the full, unnormalised set of SPCR automorphisms,
consistent with the notation introduced in the Introduction.
\end{defn}

\begin{rem}\label{rem:normalisation}
The real scaling action $L\mapsto rL$, $r\in\R^*$, preserves $\calM$ and
multiplies $\det({\bf L})$ by $r^4$. Since Lemma~\ref{lem:detL} gives
$\det({\bf L})<0$ on $\calM$, one may choose representatives with
$\det({\bf L})=-1$. This level set has real dimension $11$ and meets
each $\R^*$-scaling orbit in the two points $L$ and $-L$. The
$12$-dimensional manifold studied below is the full set $\calM$, before
quotienting by real scalings.
\end{rem}

\begin{prop}\label{prop:SPCReqs}
Let $F:M_2(\C)^2\to\C^2$ be the smooth map
\[
F(M,N)=\bigl(\det(M+\overline N),\det(M-\overline N)\bigr).
\]
Under the real-linear change of variables
\[
M=A+\overline B,\qquad N=\overline A-B,
\]
the SPCR locus $\calM$ is identified with
\[
F^{-1}(0,0)\cap\{\det({\bf L})\ne0\},
\]
where $A=\tfrac12(M+\overline N)$,
$B=\tfrac12(\overline M-N)$, and
${\bf L}=\left(\begin{smallmatrix}A&B\\\overline B&\overline A\end{smallmatrix}\right)$.
No invertibility assumption is imposed on $M$ or $N$.
\end{prop}
\begin{proof}
The change of variables is invertible over $\R$, with inverse
$A=\tfrac12(M+\overline N)$ and
$B=\tfrac12(\overline M-N)$. It gives
\[
M+\overline N=2A,\qquad M-\overline N=2\overline B,
\]
and therefore
\[
F(M,N)=\bigl(4\det A,\,4\overline{\det B}\bigr).
\]
Thus $F(M,N)=(0,0)$ is equivalent to $\det A=\det B=0$. If in addition
$\det({\bf L})\ne0$, Lemma~\ref{lem:detL} says exactly that
$\operatorname{rank}A=\operatorname{rank}B=1$, which is the SPCR
condition. This proves the proposition.
\end{proof}

\begin{rem}\label{rem:detLnonzero}
The condition $\det(\mathbf L)\neq 0$ in Proposition~\ref{prop:SPCReqs}
is essential. There exist $M,N\in\operatorname{GL}(2,\C)$ with
$F(M,N)=(0,0)$ but $\det\mathbf L=0$; see
Proposition~\ref{prop:counterexample} in \S\ref{sec:examples}.
\end{rem}

Write
$$
M=\begin{pmatrix}m_1&m_2\\m_3&m_4\end{pmatrix},
\qquad
N=\begin{pmatrix}n_1&n_2\\n_3&n_4\end{pmatrix}.
$$
Set $(\bz,\bw)=(m_1,m_2,m_3,m_4,n_1,n_2,n_3,n_4)$, regarded as the
standard complex coordinates on $M_2(\C)^2\cong\C^8$. Define
\begin{align}
H_1(\bz,\bw)&=z_1z_4-z_2z_3+\overline{w_1}\overline{w_4}-
\overline{w_2}\overline{w_3},\label{eq:H1}\\
H_2(\bz,\bw)&=z_1\overline{w_4}+\overline{w_1}z_4-
 z_2\overline{w_3}-\overline{w_2}z_3.\label{eq:H2}
\end{align}
A direct expansion gives
\[
\det(M+\overline N)=H_1(\bz,\bw)+H_2(\bz,\bw),
\qquad
\det(M-\overline N)=H_1(\bz,\bw)-H_2(\bz,\bw).
\]
Thus the zero set $F^{-1}(0,0)$ is exactly $H_1=H_2=0$. The SPCR locus in
these coordinates is the subset of this common zero set on which
$\det\mathbf L\ne0$. No invertibility condition on $M$ or $N$ is
required. The change of variables then gives
$A=\tfrac12(M+\overline N)$ and
$B=\tfrac12(\overline M-N)$, so Lemma~\ref{lem:detL} yields
$\operatorname{rank}A=\operatorname{rank}B=1$.

\begin{prop}\label{prop:dim12}
Via the real-linear identification $(A,B)\leftrightarrow(M,N)$ of
Proposition~\ref{prop:SPCReqs}, the SPCR set $\calM$ is a
$12$-dimensional real submanifold of $M_2(\C)^2\cong\C^8$.
\end{prop}
\begin{proof}
\emph{Step 1: change of coordinates.} Recall that
$M=A+\overline B$ and $N=\overline A-B$, equivalently
$A=\tfrac12(M+\overline N)$ and $B=\tfrac12(\overline M-N)$. The map
\[
\Psi:(M,N)\longmapsto(A,B)=\Bigl(\tfrac12(M+\overline N),\,\tfrac12(\overline M-N)\Bigr)
\]
is a real-linear isomorphism of $\R^{16}$ with inverse
$(A,B)\mapsto(A+\overline B,\,\overline A-B)$. Under $\Psi$,
\[
M+\overline N=2A,\qquad M-\overline N=2\overline B,
\]
hence
\[
F_1 = \det(M+\overline N)=4\det A,
\qquad
F_2 = \det(M-\overline N)=4\overline{\det B}.
\]
Therefore the zero locus $F_1=F_2=0$ is identified via $\Psi$ with
\begin{equation}\label{eq:Z}
Z:=\{(A,B)\in\C^4\times\C^4 : \det A=0,\ \det B=0\}.
\end{equation}

\emph{Step 2: $Z$ is smooth of complex codimension $2$ away from $A=0$ and $B=0$.}
The zero locus $\{\det A=0\}\subset\C^4$ of the polynomial
$\det:\C^4\to\C$ has gradient at $A$ equal to the row vector with entries
$(\operatorname{adj}A)_{ji}$; this gradient vanishes if and only if
$\operatorname{adj}A=0$, which for a $2\times 2$ matrix is equivalent to
$\operatorname{rank}A=0$, i.e.\ $A=0$. Hence $\{\det A=0\}$ is a smooth
complex hypersurface of $\C^4$ away from the origin, and the origin is its
only singular point; the same holds for $\{\det B=0\}$. Since $Z$ is the
product of these two hypersurfaces, it is a complex codimension-$2$
submanifold of $\C^4\times\C^4\cong\C^8$ away from
$(\{0\}\times\C^4)\cup(\C^4\times\{0\})$. Near every point with $A\ne 0$ and
$B\ne 0$, $Z$ is smooth of complex dimension $6$, i.e.\ real dimension $12$.

\emph{Step 3: on $\calM$, $A\ne 0$ and $B\ne 0$.} Let
$(A,B)\in Z$ satisfy $\det\mathbf L\ne 0$, where
$\mathbf L=\begin{pmatrix}A&B\\ \overline B&\overline A\end{pmatrix}$. If
$A=0$, then $\mathbf L=\begin{pmatrix}0&B\\ \overline B&0\end{pmatrix}$ and
a block row interchange brings $\mathbf L$ to
$\operatorname{diag}(\overline B,B)$, whence
$\det\mathbf L=|\det B|^2=0$ on $Z$, contradiction. Symmetrically, $B=0$
gives $\det\mathbf L=|\det A|^2=0$ on $Z$, again a contradiction. Hence
$A\ne 0$ and $B\ne 0$ on $\calM$.

\emph{Step 4: conclusion.} By Steps 1--3, $\Psi$ identifies $\calM$ with the
open subset of $Z$ defined by $A\ne 0$, $B\ne 0$, and
$\det\mathbf L\ne 0$. By Step 2, this open subset is a smooth real
$12$-dimensional submanifold of $\C^8\cong\R^{16}$. Since $\Psi$ is a real-linear isomorphism, it follows that $\calM$ is a
$12$-dimensional real submanifold of $M_2(\C)^2\cong\C^8$.
\end{proof}

\subsection{The CR structure on $\calM$}\label{ss:CR}

We regard $M_2(\C)^2\cong\C^8$ as a complex manifold with its standard
complex structure. In the coordinates $(z,w)$ above, the same two complex-valued
real-analytic functions $H_1,H_2$ from \eqref{eq:H1}--\eqref{eq:H2}
define $\calM$ by $H_1=H_2=0$, together with the open condition
$\det\mathbf L\ne0$. Write
\[
G_1=2\Re H_1,\quad G_2=2\Im H_1,\quad
G_3=2\Re H_2,\quad G_4=2\Im H_2,
\]
and set $\phi_j=\partial G_j$, $j=1,2,3,4$. Thus $\calM$ is the open
part, on which $\det({\bf L})\ne0$, of the common zero set of the four real
functions $G_1,G_2,G_3,G_4$.

\begin{prop}\label{prop:CR2}
Let $\calN=M_2(\C)^2\cong\C^8$. Let $U\subset\calM$ be the set of points
at which the four $(1,0)$-forms $\phi_1,\phi_2,\phi_3,\phi_4$ are
$\C$-linearly independent. Then:
\begin{enumerate}
\item[(a)] $U$ is open and dense in $\calM$;
\item[(b)] on $U$, the bundle
\[
T^{1,0}\calM=\{V\in T^{1,0}\calN|_U:\phi_j(V)=0,
\ j=1,2,3,4\}
\]
has complex rank $4$;
\item[(c)] the real complex-tangent bundle
$\mathcal H:=T\calM\cap J(T\calM)$ has real rank $8$, and
\[
\mathcal H\otimes\C=T^{1,0}\calM\oplus T^{0,1}\calM,
\qquad T^{0,1}\calM=\overline{T^{1,0}\calM};
\]
thus $U$ is a CR manifold of CR dimension $4$ and CR codimension $4$;
\item[(d)] this CR structure is formally integrable.
\end{enumerate}
No CR structure is asserted on $\calM\setminus U$.
\end{prop}
\begin{proof}
Introduce the $(1,0)$-forms
\begin{align*}
\alpha&=z_4\,dz_1-z_3\,dz_2-z_2\,dz_3+z_1\,dz_4,\\
\beta&=w_4\,dw_1-w_3\,dw_2-w_2\,dw_3+w_1\,dw_4,\\
\gamma&=\overline w_4\,dz_1-\overline w_3\,dz_2-
\overline w_2\,dz_3+\overline w_1\,dz_4,\\
\delta&=\overline z_4\,dw_1-\overline z_3\,dw_2-
\overline z_2\,dw_3+\overline z_1\,dw_4.
\end{align*}
Differentiating the four real defining functions gives
\[
\phi_1=\alpha+\beta,\quad
\phi_2=-i(\alpha-\beta),\quad
\phi_3=\gamma+\delta,\quad
\phi_4=-i(\gamma-\delta).
\]
Their coefficients are polynomial in the coordinates and their conjugates.
Hence the locus where the four forms are independent is open; its
complement is the common zero set of the finitely many $4\times4$ minors
of their coefficient matrix.

To see that $U$ is nonempty, take
\[
p_0=((1,0,0,1),(1,0,0,-1)),
\]
so that $M(p_0)=I_2$, $N(p_0)=\operatorname{diag}(1,-1)$,
$A(p_0)=\operatorname{diag}(1,0)$ and
$B(p_0)=\operatorname{diag}(0,1)$. Thus $p_0\in\calM$. At this point,
the $4\times4$ minor of the coefficient matrix of
$(\phi_1,\phi_2,\phi_3,\phi_4)$ in the columns
$(dz_1,dz_4,dw_1,dw_4)$ is $-16$, so $p_0\in U$.

The manifold $\calM$ is connected: by Theorem~\ref{thm:canonical}, it is
the image of the connected manifold
$\operatorname{GL}(2,\C)\times\operatorname{GL}(2,\C)$ under the continuous
map $(P,Q)\mapsto P\circ\sigma_0\circ Q$. Let $\Delta$ denote the particular minor displayed above. It is a real-analytic
function on the connected real-analytic manifold $\calM$ and is nonzero at
$p_0$. The real-analytic function $|\Delta|^2$ is nonzero at $p_0$, so the
identity principle implies that its zero set has empty interior. Since $\calM\setminus U\subseteq\{\Delta=0\}$,
the complement of $U$ has empty interior, proving density.

On $U$, independence of the four forms gives
$\dim_\C T^{1,0}\calM=8-4=4$. The real defining functions have independent
differentials there, and the $(1,0)$ tangent bundle is exactly the common
kernel displayed in the statement. Its conjugate is a rank-$4$ subbundle
of $T^{0,1}\calN$; the two bundles intersect trivially because the ambient
$(1,0)$ and $(0,1)$ tangent bundles do. Their direct sum is the complexification
of $\mathcal H=T\calM\cap J(T\calM)$, whose real rank is $8$. Since
$\dim_\R\calM=12$, the CR codimension is $12-8=4$.

Finally, if $V,W$ are local sections of $T^{1,0}\calM$, they are ambient
$(1,0)$ vector fields and satisfy $V(G_j)=W(G_j)=0$ for every $j$. Their
bracket is again an ambient $(1,0)$ vector field, and
\[
[V,W](G_j)=V(W(G_j))-W(V(G_j))=0.
\]
Thus $[V,W]$ is tangent to $\calM$ and belongs to $T^{1,0}\calM$.
This proves formal integrability.
\end{proof}

\begin{rem}\label{rem:partial-frame}
The CR structure is asserted only on the open dense subset $U$ where all
four $(1,0)$ conormal forms are independent. No extension across
$\calM\setminus U$ is asserted. The CR dimension is $4$ (not $6$): the
four real defining equations generally have four independent $(1,0)$
conormal forms.
\end{rem}

\section{Examples and explicit constructions}\label{sec:examples}

The purpose of this section is to make the preceding constructions concrete.
We exhibit (i) the canonical diagonal SPCR family with all conditions verified
by hand, (ii) a genuinely non-diagonal SPCR automorphism together with its
point on the $12$-dimensional manifold $\calM$, (iii) a counterexample showing
that the hypothesis $\det\mathbf L\neq 0$ in
Proposition~\ref{prop:SPCReqs} cannot be dropped, and (iv) the double-coset
parametrisation of $\calM$ together with the dimension count.

\subsection{The canonical diagonal family}\label{ss:canonical}

For $a,b\in\C^*$ define the real-linear map
\begin{equation}\label{eq:L_ab}
L_{a,b}(z_1,z_2)=(az_1,\,b\bar z_2).
\end{equation}
Its holomorphic and antiholomorphic blocks are
\begin{equation}\label{eq:AB_ab}
A=\begin{pmatrix} a & 0\\ 0 & 0\end{pmatrix},
\qquad
B=\begin{pmatrix} 0 & 0\\ 0 & b\end{pmatrix}.
\end{equation}
Both have rank $1$, so Lemma~\ref{lem:detL} applies. The complex matrix
associated with $L_{a,b}$ is
\[
\mathbf L_{a,b}=
\begin{pmatrix}
a & 0 & 0 & 0\\
0 & 0 & 0 & b\\
0 & 0 & \bar a & 0\\
0 & \bar b & 0 & 0
\end{pmatrix}.
\]
Its nonzero entries occupy the positions of the transposition $(2\,4)$.
Since this permutation has sign $-1$,
\[
\det\mathbf L_{a,b}=-|a|^2|b|^2<0.
\]
Thus $L_{a,b}$ is an SPCR automorphism for every $a,b\in\C^*$. Its holomorphic
and antiholomorphic directions are the coordinate axes:
\[
\overline{\ker B}=\C e_1,\qquad
\ker A=\C e_2,
\]
which are complementary, in accordance with the discussion in the
Introduction. The map $L_{a,b}$ is the composite
$L_{a,b}=D_{a,b}\circ\sigma_0$ where
$D_{a,b}=\operatorname{diag}(a,b)\in\operatorname{GL}(2,\C)$ and
$\sigma_0(z_1,z_2)=(z_1,\bar z_2)$, and it maps the real hypersurface
$\operatorname{Fix}(\sigma_0)=\C\times\R$ onto $\C\times b\R$.

\smallskip\noindent\emph{A numerical instance.} Take $a=1$, $b=i$:
\[
L(z_1,z_2)=(z_1,i\bar z_2),\qquad
A=\begin{pmatrix}1&0\\0&0\end{pmatrix},\quad
B=\begin{pmatrix}0&0\\0&i\end{pmatrix}.
\]
In the real coordinate order $(x_1,x_2,y_1,y_2)$ for the source and
$(u_1,u_2,v_1,v_2)$ for the target, one has
$u_1=x_1$, $u_2=y_2$, $v_1=y_1$, and $v_2=x_2$. Thus the real matrix of
$L$ is
\[
\begin{pmatrix}
1&0&0&0\\
0&0&0&1\\
0&0&1&0\\
0&1&0&0
\end{pmatrix},
\qquad
\det=-1.
\]
Its complex matrix is
\[
\mathbf L=
\begin{pmatrix}
1 & 0 & 0 & 0\\
0 & 0 & 0 & i\\
0 & 0 & 1 & 0\\
0 & -i & 0 & 0
\end{pmatrix},
\qquad
\det\mathbf L=-1.
\]
The corresponding point on the manifold is
$(M,N)=\bigl(\operatorname{diag}(1,-i),\operatorname{diag}(1,-i)\bigr)
\in\calM$; one checks directly that $M+\bar N=\operatorname{diag}(2,0)$ and
$M-\bar N=\operatorname{diag}(0,-2i)$, so $F(M,N)=(0,0)$.

\subsection{A non-diagonal SPCR automorphism}\label{ss:nondiagonal}

The diagonal examples are geometrically transparent but special: the
holomorphic and antiholomorphic directions are the coordinate axes. We now
give an SPCR automorphism whose directions are not aligned with the
coordinates.

\begin{prop}\label{prop:nondiag}
Let
\[
M=\begin{pmatrix}1&2\\ 1&-i\end{pmatrix},
\qquad
N=\begin{pmatrix}1&2\\ -1&-i\end{pmatrix}.
\]
Then $F(M,N)=(0,0)$ and $\det\mathbf L\ne0$, so the associated pair
represents a point of $\calM$. (In this example both $M$ and $N$ happen
to be invertible.) The associated SPCR automorphism is
\begin{equation}\label{eq:L_nondiag}
L(z_1,z_2)=\bigl(z_1+2z_2,\ \bar z_1+i\bar z_2\bigr).
\end{equation}
Its holomorphic direction is $\overline{\ker B}=\C(-i,-1)$. Its
antiholomorphic direction is $\ker A=\C(2,-1)$. These two lines are
transverse.
\end{prop}

\begin{proof}
First, $\det M=-i-2=-(2+i)\neq 0$ and $\det N=-i+2=2-i\neq 0$, so
$M,N\in\operatorname{GL}(2,\C)$. Since $\bar N=\begin{pmatrix}1&2\\ -1&i\end{pmatrix}$,
\[
M+\bar N=\begin{pmatrix}2&4\\ 0&0\end{pmatrix},
\qquad
M-\bar N=\begin{pmatrix}0&0\\ 2&-2i\end{pmatrix},
\]
both of determinant zero; hence $F(M,N)=(0,0)$.

Next, $A=\tfrac12(M+\bar N)=\begin{pmatrix}1&2\\ 0&0\end{pmatrix}$ and,
since $\bar M=\begin{pmatrix}1&2\\ 1&i\end{pmatrix}$,
$B=\tfrac12(\bar M-N)=\begin{pmatrix}0&0\\ 1&i\end{pmatrix}$. Thus
$\operatorname{rank}A=\operatorname{rank}B=1$.

To check $\det\mathbf L\neq 0$ we compute directly:
\[
\mathbf L=
\begin{pmatrix}
1 & 2 & 0 & 0\\
0 & 0 & 1 & i\\
0 & 0 & 1 & 2\\
1 & -i & 0 & 0
\end{pmatrix}.
\]
Expanding along the first column gives $\det\mathbf L=-5$. This verifies directly that the real-linear map is invertible; its negative
determinant also agrees with the sign conclusion of Lemma~\ref{lem:detL}.

Finally, $\ker A=\{(z_1,z_2):z_1+2z_2=0\}=\C(2,-1)$ and
$\ker B=\{(z_1,z_2):z_1+iz_2=0\}=\C(i,-1)$, so
$\overline{\ker B}=\C(-i,-1)$. The two directions are transverse since
$\det\begin{pmatrix}-i&2\\ -1&-1\end{pmatrix}=i+2\neq 0$.

The formula \eqref{eq:L_nondiag} is read off from $w=Az+B\bar z$:
$w_1=z_1+2z_2$ and $w_2=\bar z_1+i\bar z_2$.
\end{proof}

\smallskip
\emph{Equivalence to the canonical form.} Choose
$R=\begin{pmatrix}-i&2\\ -1&-1\end{pmatrix}$, whose columns are the holomorphic
and antiholomorphic directions; $\det R=i+2\neq 0$. Set $z=R\zeta$. Then
\[
A'=AR=\begin{pmatrix}-(2+i)&0\\ 0&0\end{pmatrix},
\qquad
B'=B\bar R=\begin{pmatrix}0&0\\ 0&2-i\end{pmatrix},
\]
so in the $\zeta$-coordinates $L$ takes the canonical diagonal form
\[
L'(\zeta_1,\zeta_2)=\bigl(-(2+i)\zeta_1,\ (2-i)\bar\zeta_2\bigr),
\]
recovering \eqref{eq:L_ab} with $a=-(2+i)$ and $b=2-i$. The change of
coordinates multiplies $\det\mathbf L$ by $|\det R|^2=5$, giving
$\det\mathbf L'=-25=-|a|^2|b|^2$. This is an explicit instance
of Theorem~\ref{thm:canonical}.

\subsection{Why the condition $\det\mathbf L\neq 0$ is necessary}

Proposition~\ref{prop:SPCReqs} characterises $\calM$ as the open subset of
$F^{-1}(0,0)$ on which $\det\mathbf L\neq 0$. The following example shows that
this last condition cannot be omitted: there exist
$M,N\in\operatorname{GL}(2,\C)$ with $F(M,N)=(0,0)$ but $\det\mathbf L=0$, so
that no automorphism $L$ is associated with $(M,N)$.

\begin{prop}\label{prop:counterexample}
Let
\[
M=\begin{pmatrix}i&-i\\ -1&-1\end{pmatrix},
\qquad
N=\begin{pmatrix}-i&-i\\ -1&1\end{pmatrix}.
\]
Then $M,N\in\operatorname{GL}(2,\C)$, $F(M,N)=(0,0)$, but $\det\mathbf L=0$;
consequently $(M,N)\notin\calM$.
\end{prop}

\begin{proof}
$\det M=-i-i=-2i\neq 0$ and $\det N=-i-i=-2i\neq 0$, so $M,N$ are
invertible. Since $\bar N=\begin{pmatrix}i&i\\ -1&1\end{pmatrix}$,
\[
M+\bar N=\begin{pmatrix}2i&0\\ -2&0\end{pmatrix},
\qquad
M-\bar N=\begin{pmatrix}0&-2i\\ 0&-2\end{pmatrix},
\]
both of determinant zero; hence $F(M,N)=(0,0)$.

The associated matrices are
\[
A=\tfrac12(M+\bar N)=\begin{pmatrix}i&0\\ -1&0\end{pmatrix},
\qquad
B=\tfrac12(\bar M-N)=\begin{pmatrix}0&i\\ 0&-1\end{pmatrix},
\]
both of rank $1$. But now
\[
\mathbf L=
\begin{pmatrix}
i & 0 & 0 & i\\
-1 & 0 & 0 & -1\\
0 & -i & -i & 0\\
0 & -1 & -1 & 0
\end{pmatrix},
\]
and the second row equals $i$ times the first, so $\det\mathbf L=0$.
Therefore $(M,N)$ corresponds to no linear automorphism of $\R^4$, and in
particular not to an SPCR automorphism.
\end{proof}

The example shows that the nonvanishing condition
$\det\mathbf L\ne0$ in Proposition~\ref{prop:SPCReqs} is essential: the
rank-one equations alone can also hold at a singular real-linear map, which
is not an automorphism.

\subsection{The double coset and the dimension count}

Let $\sigma_0\in\operatorname{GL}(4,\R)$ be the mixed real-linear involution
$\sigma_0(z_1,z_2)=(z_1,\bar z_2)$, whose fixed locus
$\C\times\R\subset\C^2$ is a real hypersurface. For every
$P,Q\in\operatorname{GL}(2,\C)$, the composite
$L=P\circ\sigma_0\circ Q$ is an SPCR automorphism: its two complex blocks
have rank $1$, and
\[
\det({\bf L})=-|\det P|^2|\det Q|^2<0.
\]
Conversely, by Theorem~\ref{thm:canonical}, every SPCR automorphism arises
in this way. Thus the map
\[
\operatorname{GL}(2,\C)\times\operatorname{GL}(2,\C)
\longrightarrow\calM,\qquad (P,Q)\longmapsto P\circ\sigma_0\circ Q
\]
is surjective. The stabiliser of $\sigma_0$ for this left-right action consists precisely
of the diagonal pairs
\[
P=\operatorname{diag}(p_1,p_2),\qquad
Q=\operatorname{diag}(p_1^{-1},\overline{p_2}^{\,-1}),
\qquad p_1,p_2\in\C^*.
\]
Indeed, $P\sigma_0Q=\sigma_0$ forces $Q$ to preserve the two coordinate
axes, and direct substitution gives the displayed relation. The stabiliser
has real dimension $4$. Consequently
\[
\dim_\R\calM=16-4=12,
\]
in agreement with Proposition~\ref{prop:dim12}.

The real scaling action $L\mapsto rL$ ($r\in\R^*$) has one-dimensional
orbits. Its quotient therefore has real dimension $11$. The level set
$\det({\bf L})=-1$ is a two-to-one cross-section for this action (the two
representatives are $L$ and $-L$); it has real dimension $11$. We do not take this quotient in the definition of
$\calM$; the CR structure of Proposition~\ref{prop:CR2} is defined on an
open dense subset of the full $12$-dimensional manifold. The diagonal family
\eqref{eq:L_ab} is a real $4$-dimensional subfamily of $\calM$, and its
intersection with the level set $\det({\bf L})=-1$ is $3$-dimensional.
The non-diagonal example of Proposition~\ref{prop:nondiag} is not in the
standard-coordinate diagonal family, although it is equivalent to a
canonical diagonal map by complex-linear changes of coordinates.

\section{Concluding remarks and open problems}\label{sec:outlook}

The results of this paper are the algebraic characterisation of SPCR
automorphisms (Lemma~\ref{lem:detL}), their canonical form
(Theorem~\ref{thm:canonical}), their description as the open subset
$F^{-1}(0,0)\cap\{\det\mathbf L\neq 0\}$ of $M_2(\C)^2$
(Proposition~\ref{prop:SPCReqs}), the identification of this subset with a
$12$-dimensional submanifold $\calM$ (Proposition~\ref{prop:dim12}), and the
induced CR structure of CR dimension $4$ and CR codimension $4$ on the open
dense subset $U\subset\calM$ (Proposition~\ref{prop:CR2}). Several statements in the
Introduction point, in a motivational sense, to further questions that are
\emph{not} resolved in the present paper. We collect them here so that the
scope of the paper is unambiguous.

\begin{enumerate}
\item \emph{Extension across $\calM\setminus U$.} The CR structure is
defined only on the open dense subset $U$ of $\calM$ (which is
real-algebraically open in these coordinates); on
$\calM\setminus U$ the four forms $\phi_1,\phi_2,\phi_3,\phi_4$ may become dependent, and the
fibre dimension of the kernel may jump. Whether the CR structure admits a
canonical extension as a singular CR structure, or whether
$\calM\setminus U$ can itself be resolved by a blow-up that carries a CR
structure, is an open problem.

\item \emph{Levi form and pseudoconvexity.} Proposition~\ref{prop:CR2}
provides the CR distribution on $U$; its Levi form, signature, and the pseudoconvexity or Levi-flatness of $U$ are not
computed here. A natural problem is to determine whether the Levi form is
non-degenerate on some open subset of $U$, and whether the CR structure is
strictly pseudoconvex.

\item \emph{Chains.} The chains associated with the CR structure on $U$ are natural invariants of $U$ as a CR manifold. Their behaviour under the
$\operatorname{GL}(2,\C)^2$ action and their interpretation in terms of the
underlying CR geometry of $\C^2$ are open.

\item \emph{Relation to moduli of CR maps.} The manifold $\calM$ may be
regarded informally as a finite-dimensional analogue of the moduli space of
CR maps between real hypersurfaces. Making this analogy precise would require
identifying $\calM$ as the linear part of a suitable jet space of partially
CR maps, and comparing the corresponding deformation theories; this is
beyond the scope of the present paper.

\item \emph{Reflection-principle consequences.} The examples of
\S\ref{sec:examples} show that SPCR automorphisms exist in abundance and are
not restricted to the diagonal normal form. Whether SPCR automorphisms can be
used as local linear models near points of a real hypersurface where a CR map
is holomorphic in one direction and antiholomorphic in another---and whether
the reflection principle of Lewy \cite{Lewy} and Pinchuk \cite{Pinchuk}
extends to such partially CR maps---is an open problem. We do not attempt it
here.

\item \emph{Non-linear SPCR maps.} We have treated only the linear case. A
natural next step is the classification of non-linear smooth maps
$F:D_1\to D_2$ between domains in $\C^2$ whose differential has the SPCR rank
pattern everywhere. Whether such maps exist beyond the trivially linear ones,
and whether a rigidity phenomenon analogous to that of holomorphic maps
holds, is an open problem.

\item \emph{Para-CR interpretation.} As noted in the Introduction, the interaction of the two structures $(J,-J)$ suggests possible connections with para-CR geometry; making the SPCR
condition into a statement about a para-CR version of the classical
Cartan--Tanaka theory is an open direction.
\end{enumerate}

{\it Declaration:} The author was funded by the Medicus Programme No. 83765.

\bibliographystyle{amsplain}

\begin{thebibliography}{10}

\bibitem{BER} M.~S. Baouendi, P.~Ebenfelt and L.~P. Rothschild, \emph{Real Submanifolds in Complex Space and Their Mappings}, Princeton Mathematical Series, vol.~47, Princeton University Press, Princeton, NJ, 1999.

\bibitem{Boggess}
A.~Boggess, \emph{CR Manifolds and the Tangential Cauchy-Riemann Complex},
Studies in Advanced Mathematics, CRC Press, Boca Raton, FL, 1991.

 \bibitem{ChernMoser} S.~S. Chern and J.~K. Moser, \emph{Real hypersurfaces in complex manifolds}, Acta Math. \textbf{133} (1974), 219--271.

\bibitem{DragomirTomassini}
S.~Dragomir and L.~Tomassini, \emph{Differential Geometry and Analysis on CR
Manifolds}, Progress in Mathematics, vol. 246, Birkh\"auser Boston, Inc.,
Boston, MA, 2006.

\bibitem{Jacobowitz} H.~Jacobowitz, \emph{An Introduction to CR Structures}, Mathematical Surveys and Monographs, vol.~32, American Mathematical Society, Providence, RI, 1990.

\bibitem{Lewy}
H.~Lewy, \emph{On the local character of the solutions of an elliptic
differential equation of the fourth order in three variables}, Bull. Amer.
Math. Soc. \textbf{43} (1937), 689--692.

\bibitem{Pinchuk}
S.~I. Pinchuk, \emph{On the analytic continuation of holomorphic mappings},
Mat. Sb. (N.S.) \textbf{98(140)} (1975), 416--435; English transl. in Math.
USSR-Sb. \textbf{27} (1975).

\bibitem{Webb}
S.~M. Webster, \emph{On the mapping problem for algebraic real hypersurfaces},
Invent. Math. \textbf{43} (1977), no. 1, 53--68.

 

\end{thebibliography}

\end{document}